\documentclass[10pt]{article}
\usepackage[total={6in,9.5in}]{geometry}
\usepackage{amsmath,amssymb,amsthm}
\usepackage{siunitx}
\usepackage{optidef}
\usepackage{palatino}
\usepackage{graphicx,epstopdf,epsfig}
\usepackage[dvipsnames]{xcolor}
\usepackage{float}
\usepackage[caption=false,font=footnotesize]{subfig}
\usepackage{url}
\usepackage{microtype}
\usepackage[normalem]{ulem}
\usepackage[section]{placeins} 
\usepackage{centernot}
\usepackage{enumitem}
\usepackage[noadjust]{cite}

\newtheorem{theorem}{Theorem}[section]
\newtheorem{proposition}{Proposition}[section]
\newtheorem{definition}{Definition}[section]
\newtheorem{corollary}{Corollary}[section]
\newtheorem{lemma}{Lemma}[section]

\newcommand{\blE}{\mathbb{E}}

\newcommand{\Sy}{\mathbb{S}}
\renewcommand{\Re}{\mathbb{R}}

\newcommand{\Vr}{\mathbb{V}}

\newcommand{\cN}{\mathcal{N}}

\newcommand{\cR}{\mathcal{R}}

\def\A{\mathbf{A}}
\def\B{\mathbf{B}}

\def\D{\mathbf{D}}

\def\I{\mathbf{I}}

\def\K{\mathbf{K}}

\def\Q{\mathbf{Q}}

\def\W{\mathbf{W}}

\def\b{\boldsymbol{b}}
\def\c{\boldsymbol{c}}

\def\e{\boldsymbol{e}}

\def\v{\boldsymbol{v}}
\def\w{\boldsymbol{w}}
\def\x{\boldsymbol{x}}
\def\y{\boldsymbol{y}}
\def\z{\boldsymbol{z}}
\def\ze{\mathbf{0}}

\newcommand{\prox}{\mathrm{prox}}
\newcommand{\fix}{\mathrm{fix}\,}

\DeclareMathOperator*{\diag}{diag}

\begin{document}

\markboth{Submitted in IEEE Signal Processing Letters}{}

\title{On the Strong Convexity of \\ PnP Regularization using Linear Denoisers}

\date{}

\author{Arghya~Sinha  and  Kunal~N.~Chaudhury
\thanks{A.~Sinha was supported by a PMRF Fellowship TF/PMRF-22-5534 and K.~N.~Chaudhury was supported by grants CRG/2020/000527 and STR/2021/000011 from the Government of India.}
}

\maketitle

\begin{abstract}
In the Plug-and-Play (PnP) method, a denoiser is used as a regularizer within classical proximal algorithms for image reconstruction. It is known that a broad class of linear denoisers can be expressed as the proximal operator of a convex regularizer. Consequently, the associated PnP algorithm can be linked to a convex optimization problem $\mathcal{P}$. {For such a linear denoiser}, we prove that $\mathcal{P}$ exhibits strong convexity for linear inverse problems. Specifically, we show that the strong convexity of $\mathcal{P}$ can be used to certify objective and iterative convergence of \textit{any} PnP algorithm derived from classical proximal methods. 
\end{abstract}
 

\section{Introduction}

{There has been a growing interest in applying denoisers as regularizers for image reconstruction \cite{sreehari2016plug,hurault2022gradient,cohen2021regularization,romano2017little,reehorst2018regularization,zhu2023denoising,terris2024equivariant,pesquet_learning_2021,hertrich_convolutional_2021,goujon_learning_2024}}. The archetype framework in this regard is the Plug-and-Play (PnP) method \cite{sreehari2016plug}, wherein regularization is performed by plugging a denoiser into classical proximal algorithms. {The success of PnP has spurred interest in its stability aspects~\cite{sreehari2016plug,CWE2017,Ryu2019_PnP_trained_conv,Teodoro2019PnPfusion,reehorst2018regularization,pesquet_learning_2021,terris2024equivariant,hurault2022gradient,hertrich_convolutional_2021,proximal-hurault22a}. Indeed, while trained denoisers perform exceedingly well in practice, they can at times cause the iterations to diverge~\cite{reehorst2018regularization,nair2024averaged,terris2024equivariant}. Thus, having a convergence guarantee can act as a safeguard~\cite{hurault2022gradient,hertrich_convolutional_2021,pesquet_learning_2021}.}

In this note, we are concerned with linear denoisers \cite{milanfar2013tour,milanfar2013symmetrizing} that perform well in different applications \cite{nair2019hyperspectral,nair2022plug} and yet are analytically tractable \cite{sreehari2016plug,Teodoro2019PnPfusion,gavaskar2021plug,nair2021fixed}. Specifically, it was shown in \cite{sreehari2016plug} that certain linear denoisers can be conceived as the proximal operator of a convex regularizer $\varphi$. Consequently, we can view PnP as a classical proximal algorithm
applied to a convex optimization problem of the form
\begin{equation}
\label{eq:recon}
\min_{\x \in \Re^n} \ \ell(\x) + \lambda\, \varphi(\x) 
\end{equation}
for some $\lambda >0$, where $\ell$ is the model-based loss function. This forms the basis of the convergence analysis in prior works~\cite{sreehari2016plug,Teodoro2019PnPfusion,gavaskar2021plug}. 

\textbf{Contributions}. In this note, we prove that the objective function in \eqref{eq:recon} is strongly convex for linear inverse problems. This is a nontrivial observation since the components $\ell$ and $\varphi$ are not strongly convex. Additionally, we show that the strong convexity result can be extended to nonsymmetric denoisers, provided we work with the so-called scaled PnP algorithms in \cite{gavaskar2021plug}. Strong convexity guarantees that \eqref{eq:recon} has exactly one minimizer \cite{beck2017book}. Thus, strong convexity does away with the need to \textit{assume} that \eqref{eq:recon} has minimizers, reinforcing existing convergence results. Notably, we are able to certify objective and iterate convergence of \textit{any} PnP algorithm derived from a classical proximal method (Thm.~\ref{thm:convg}).

\textbf{Organization}. We provide some necessary background on PnP in Sec.~\ref{backgnd} before
establishing the strong convexity of \eqref{eq:recon} for both symmetric and nonsymmetric denoisers in Sec.~\ref{results}. We show how this can be used to guarantee objective and iterate convergence of PnP algorithms in Sec.~\ref{convg}. Numerical results on strong convexity are presented in Sec.~\ref{exp}.

\textbf{Notations}. We mostly follow the notations in \cite{beck2017book}. We will view $\Re^n$ (the space of grayscale images) as an Euclidean space $\blE$, with an appropriate inner product $\langle \cdot, \cdot \rangle$ and norm $\| \cdot \|$ induced by $\langle \cdot, \cdot \rangle$. We denote this as a tuple $(\blE, \langle \cdot, \cdot \rangle)$ since we can have two different inner products for the same $\blE$. The set of $(-\infty,\infty]$-valued functions on $\blE$ that are proper, closed, and convex is denoted by $\Gamma_0(\blE)$. The indicator function of a nonempty convex set  $\Omega \subset \blE$ is the function $\iota_{\Omega} \in \Gamma_0(\blE)$ given by $\iota_{\Omega}(\x) = 0$ for $\x \in \Omega$, and $\iota_{\Omega}(\x) =+\infty$ for $\x \notin \Omega$. We use $\fix (\W)$ to denote the fixed points of $\W$, $\cN(\A)$ and $\cR(\A)$ for the null space and the range space of $\A$, and $\sigma(\B)$ for the spectrum of $\B$. We use $\Sy^n$ (resp.~$\Sy^n_+$) for the set of symmetric (resp.~positive semidefinite) matrices in $\Re^{n \times n}$ and $\mathcal{S}_n$ for the unit sphere in $\Re^n$. {We use $\W^{\dagger}$ for the pseudoinverse of $\W$.}

\section{Background}
\label{backgnd}

Any PnP algorithm has two basic components: the loss function and the denoiser. We will work with linear inverse problems, where the task is to recover an unknown image $\boldsymbol{\xi} \in \Re^n$ from linear measurements $\b = \A \boldsymbol{\xi}  + \w$,
where $\A \in \Re^{m \times n}$ is the forward operator, $\b \in \Re^m$ is the observed image, and $\w \in \Re^m$ is white {Gaussian} noise. The standard loss function in this case is
\begin{equation}
\label{loss}
\ell_{\A}(\x) =  \frac{1}{2}\| \A \x - \b \|_2^2.
\end{equation}
This model includes inpainting, deblurring, and superresolution. The forward operator $\A$ is a sampling operator in inpainting, a blur (convolution) in deblurring, and composition of blur and subsampling in superresolution \cite{bouman2022foundations}. On the other hand, we will work with linear denoisers such as NLM \cite{buades2005non}, symmetric NLM \cite{sreehari2016plug}, LARK \cite{takeda2007kernel}, GLIDE \cite{talebi2013global}, and GMM \cite{Teodoro2019PnPfusion}. We refer the reader to \cite{milanfar2013tour,milanfar2013symmetrizing} and \cite[Sec.~IV]{sreehari2016plug} for more information on symmetric and kernel denoisers. 

In PnP algorithms, regularization is performed by plugging a denoiser into a proximal algorithm such as ISTA, FISTA, HQS, or ADMM \cite{sreehari2016plug,Ryu2019_PnP_trained_conv,gavaskar2021plug,zhu2023denoising}. For example, starting with an initial guess $\x_0 \in \Re^n$, the update in PnP-ISTA involves a gradient step on the loss followed by a denoising (regularization) step:
\begin{equation}
\label{eq:pnp-ista}
\x_{k+1} = \W\big(\x_k- \gamma \nabla \ell_{\A}(\x_k)\big),
\end{equation}
where  $\gamma$ is the step size and $\W$ is the denoiser.

It was observed in \cite{sreehari2016plug} that for symmetric denoisers such as  DSG-NLM \cite{sreehari2016plug} and GMM \cite{Teodoro2019PnPfusion}, $\W$ can be expressed as the proximal operator of a convex function~\cite{beck2017book}. 
\begin{definition}
\label{defProx}
Consider an Euclidean space $(\blE, \langle \cdot, \cdot \rangle)$ and let $\varphi \in \Gamma_0(\blE)$. The proximal operator of $\varphi$ is defined as
\begin{equation}
(\forall \, \x \in \blE) \quad \prox_{\varphi, \| \cdot \|}(\x) = \underset{\z \in \blE}{\arg \min} \ \frac{1}{2} \| \z - \x \|^2 + \varphi(\z).
\end{equation}
\end{definition}

It was shown in \cite{sreehari2016plug} that there exists {$\varphi_\W \in \Gamma_0(\Re^n)$ such that $\W = \prox_{\varphi_{\W},\| \cdot \|_2}$.} An exact formula for $\varphi_\W$ was later reported in \cite{Teodoro2019PnPfusion}. However, it will be convenient for us to work with the following reformulation.
\begin{proposition}
\label{prop:formula_prox}
Let $\W \in \Sy^n$ and $\sigma(\W) \subset [0,1]$. Define $\varphi_{\W} \in \Gamma_0(\Re^n)$ to be 
\begin{equation}
\label{formula_prox}
\varphi_{\W}(\x) = \frac{1}{2} \x^\top  \W^{\dagger} (\I-\W) \x + \iota_{\cR(\W)} (\x).
\end{equation}
Then $\W = \prox_{\varphi_{\W},\| \cdot \|_2}$.
\end{proposition}

A proof of Proposition~\ref{prop:formula_prox} is provided in the Appendix for completeness. In other words, if $\W \in \Sy^n$ and $\sigma(\W) \subset [0,1]$, we can view PnP-ISTA as the standard ISTA algorithm applied to problem \eqref{eq:recon}, where the regularizer $\varphi=\varphi_{\W}$ is given by \eqref{formula_prox}. This applies not only to ISTA but to any PnP algorithm based on proximal methods, such as FISTA, ADMM, and HQS.

On the other hand, kernel denoisers such as NLM are not symmetric and hence cannot be expressed as a proximal operator. This comes as a consequence of a classical result due to Moreau \cite{Moreau1965}. Customized to linear operators, the result is as follows.

\begin{theorem}
\label{Moreau}
Consider the Euclidean space $(\Re^n,\langle \cdot, \cdot \rangle_2)$ with inner product $\langle \x, \y \rangle_2 = \x^\top \! \y$. A linear operator $\W$ on $(\Re^n,\langle \cdot, \cdot \rangle_2)$ is a proximal operator if and only if $\W \in \Sy^n$ and $\sigma(\W) \subset [0,1]$.
\end{theorem}

Motivated by this result, it was observed in \cite{gavaskar2021plug} that if $\W$ is a nonsymmetric kernel denoiser, we must work with an inner product different from the standard inner product (dot product) on $\Re^n$. More precisely, consider a kernel denoiser of the form $\W=\D^{-1}\, \K$, where the kernel matrix $\K \in \Sy^n_+$ is nonnegative with positive diagonal elements, and $\D= \diag(\K \e)$ is the normalizing matrix \cite{milanfar2013tour}. {It was observed in \cite{gavaskar2021plug} that $\W$ is a proximal operator of some $\varphi\in\Gamma_0(\Re^n)$ (also see Proposition~\ref{prop:proxkernelW}) in $(\Re^n,\langle \cdot, \cdot \rangle_{\D})$,} where
  \begin{equation}
    \label{ip}
    \langle \x, \y \rangle_{\D} := \x^\top \D \y.
  \end{equation}
Accordingly, the gradient in \eqref{eq:pnp-ista} should be changed to 
\begin{equation}
\label{gradD}
\nabla_{\D}  \ell_{\A} (\x)= \D^{-1} \nabla  \ell_{\A}(\x).
\end{equation}

The point is that the gradient depends on the choice of the inner product, and \eqref{gradD} is the gradient with respect to the inner product \eqref{ip}. The other updates remain unchanged. Following \cite{gavaskar2021plug}, we will refer to these as scaled-PnP-ISTA. Similarly, by replacing the gradient with \eqref{gradD}, we get scaled versions of FISTA, ADMM, and HQS. In Sec.~\ref{results}, we establish the objective and iterate convergence of these algorithms, which was left as an open problem in \cite{ACK2023-contractivity}.

\section{Strong Convexity of PnP}
\label{results}

We begin by defining strong convexity in the general setting of an Euclidean space \cite{beck2017book}.

\begin{definition}
\label{def:}
A function $f \in \Gamma_0(\blE)$ is said to be strongly convex on the Euclidean space $(\blE, \langle \cdot, \cdot \rangle)$ if there exists $\mu >0$ such that, for all $\x,\y \in \blE$ and $\theta \in [0,1]$,
\begin{align}
\label{def:SC}
f( \theta \x + (1-\theta) \y) \leqslant \theta f(\x) &+ (1-\theta) f(\y) \nonumber \\
& - \frac{\mu}{2}\theta (1-\theta) \| \x-\y\|^2.
\end{align}
\end{definition}

The largest $\mu$ in \eqref{def:SC} is called the strong convexity index of $f$. It follows from \eqref{def:SC} that $f$ is $\mu$-strongly convex if and only if $\x \mapsto f(\x) - (\mu/2) \| \x \|^2 \in \Gamma_0(\blE)$. Based on this observation, we can conclude that the nonempty level sets of $f$ are closed and bounded. Consequently, we obtain the following important result (see \cite{beck2017book} for a proof).
{\begin{theorem}
\label{thm:existenceunique}
Suppose $f \in \Gamma_0(\blE)$ is strongly convex on $(\blE, \langle \cdot, \cdot \rangle)$. Then there exists an unique $\x^*\in \blE$ such that $\x^* = \arg\min_{\x \in \blE} \, f(\x)$.
\end{theorem}}
\subsection{Symmetric denoiser}
We first establish strong convexity of \eqref{eq:recon} for $\varphi=\varphi_{\W}$, where $\W$ is a symmetric denoiser. We make a general observation in this regard.
\begin{lemma}
\label{lemma1}
Let $\Q \in \Sy^n_+$, $\c \in \Re^n, d \in \Re,$ and $\Vr$ be a subspace of $\Re^n$. Consider the function $f \in \Gamma_0(\Re^n)$ given by
\begin{equation}
\label{formula:f}
f(\x) = \frac{1}{2}\x^\top \! \Q \x + \c^\top \x +d + \iota_{\Vr}(\x).
\end{equation}
Then $f$ is strongly convex on $(\Re^n, \langle \cdot, \cdot \rangle_2)$ if and only if $\cN (\Q) \cap \Vr = \{\ze\}$. Moreover, the strong convexity index of $f$ is
\begin{equation}
\label{def:mu}
\mu = \inf \, \left\{\v^\top \Q \v: \ \v \in \Vr \cap \mathcal{S}_n \right\}.
\end{equation}
\end{lemma}

\begin{proof}
Suppose $f$ is strongly convex (with index $\mu>0$) and there exists $\v \in \cN (\Q) \cap \Vr, \, \v \neq \ze$. Letting $\theta=1/2, \x=\v,$ and $\y=\ze$ in \eqref{def:SC}, we arrive at the impossible result:
\begin{equation*}
\frac{1}{2}\c^\top \!\v \leqslant \frac{1}{2}\c^\top \!\v - \frac{1}{8} \mu \|\v\|_2^2
\end{equation*}
Thus, we must have $\cN (\Q) \cap \Vr = \{\ze\}$ if $f$ is strongly convex. 

Conversely, assume that $\cN (\Q) \cap \Vr = \{\ze\}$. We assert that $\mu > 0$. Indeed, if $\mu=0$, we would have  $\v^\top \! \Q\v=0$ for some $\v_0 \in \Vr\cap \mathcal{S}_n$ as the set $\left\{\v^\top \Q \v: \ \v \in \Vr \cap \mathcal{S}_n \right\}$ is closed. Moreover, as $\Q \in \Sy^n_+$, we would have $\v_0 \in \cN(\Q)$, producing a nonzero $\v_0 \in \cN (\Q) \cap \Vr$, contrary to our assumption.  
To verify \eqref{def:SC}, note that it suffices to work with $\x, \y \in \Vr$, the effective domain of $f$. Now, for all  $\x, \y \in \Vr$ and $\theta \in [0,1]$, a simple calculation gives us
\begin{align*}
\theta f(\x) & + (1-\theta) f(\y)  - f( \theta \x + (1-\theta) \y)  \\
& =\frac{1}{2} \theta (1-\theta)  (\x-\y)^\top \! \Q (\x-\y).
\end{align*}
Since $\Vr$ is a subspace of $\Re^n$, $\x-\y \in \Vr$. Hence, we have from \eqref{def:mu} that $ (\x-\y)^\top \! \Q (\x-\y) \geqslant \mu \|\x-\y\|_2^2$. This completes the verification of \eqref{def:SC}. It is also evident from the above analysis that the strong convexity index of $f$ is $\mu$.
\end{proof}

{We are now ready to establish strong convexity of \eqref{eq:recon} when the loss function is \eqref{loss} and the regularizer is \eqref{formula_prox}.}
\begin{theorem}
\label{thm:SC} 
Let $\A \in \Re^{m \times n}$ and $\W \in \Sy^n$, where $\sigma(\W) \subset [0,1]$. Consider the function $f \in \Gamma_0(\Re^n)$ given by
\begin{equation}
\label{obj}
f(\x) =\ell_{\A}(\x) + \lambda\, \varphi_{\W}(\x) 
\end{equation}
{where $\ell_{\A}$  and $\varphi_\W$ are given by  \eqref{loss} and \eqref{formula_prox}}. Then $f$ is strongly convex on $(\Re^n, \langle \cdot, \cdot \rangle_2)$ for any $\lambda >0$ if and only if  $\cN(\A)  \cap \fix (\W) = \{\ze\}$. Moreover, if
\begin{equation}
\label{def:Q}
\Q:=\A^\top \! \A + \lambda \W^{\dagger}  (\I-\W),
\end{equation}
then the strong convexity index of $f$ is
\begin{equation*}
\mu = \inf \, \left\{\v^\top \!  \Q \v: \ \v \in \cR(\W) \cap \mathcal{S}_n \right\}.
\end{equation*}
\end{theorem}

\begin{proof}
  {To use Lemma~\ref{lemma1}, we write \eqref{obj} as in \eqref{formula:f}, where $\Q$ is given by \eqref{def:Q}, $\c = \boldsymbol{0}, d = 0$, and $\Vr = \cR(\W)$}. The desired result would follow from Lemma~\ref{lemma1} if we can show that $\cN (\Q) \cap \Vr  = \{\ze\}$.

Now, since $\W \in \Sy^n$ and $\sigma(\W) \subset [0,1]$, we have $\W^\dagger \in \Sy^n$ and $\sigma(\W^\dagger) \subset [0,1]$, so that $\W^{\dagger}  (\I-\W) \in \Sy^n_+$. In particular, this means $\Q \in \Sy^n_+$, and
\begin{align}
\label{eq3}
\cN(\Q)=\cN(\A) \cap \cN(\W^{\dagger}  (\I-\W)).
\end{align}
Therefore, 
\begin{align*}
\cN (\Q) \cap  \Vr   = \cN(\A) \cap \cN(\W^{\dagger}  (\I-\W)) \cap \cR(\W).
\end{align*}
We claim that $\cN(\W^{\dagger}  (\I-\W)) \cap \cR(\W)=\fix(\W)$. Indeed, if $\W\x=\x$, then  trivially $\W^{\dagger}  (\I-\W)\x=\ze $ and $\x \in \cR(\W)$. In the other direction, if $\W^{\dagger}  (\I-\W)\x=\ze$ and $\x = \W \z$ for some $\z \in \Re^n$, we would have $\W^{\dagger}\W\z= \W^{\dagger}\W^2 \z = \W\z=\x$. Multiplying this by $\W$, we get $\x=\W\x$, establishing our claim. Combining this with \eqref{eq3}, we have $\cN (\Q) \cap  \Vr = \cN(\A) \cap \fix(\W) = \{\ze\}$, establishing our claim.
\end{proof}

We apply Thm.~\ref{thm:SC} to inpainting, deblurring, and superresolution. We assume that at least one pixel is sampled for inpainting, and the blur is nonnegative (and nonzero) for deblurring and superresolution; these conditions are typically met in practice. 

\begin{corollary}
\label{cor:sym}
Let $\W \in \Sy^n_+$ be stochastic and irreducible, and let $\A$ be the forward model corresponding to inpainting, deblurring, or superresolution. Then optimization \eqref{eq:recon} with $\ell=\ell_{\A}$ and $\varphi=\varphi_{\W}$ is strongly convex on $(\Re^n, \langle \cdot, \cdot \rangle_2)$.
\end{corollary}

\begin{proof}
Since $\W \in \Sy^n_+$ and $\W\e=\e$ ($\e$ is the all-ones vector), we have $\sigma(\W) \subset [0,1]$. Thus, $\W$ satisfies the hypothesis in Thm.~\ref{thm:SC}. On the other hand, we can conclude from $\W\e=\e$ that $\mathrm{span}\{\e\} \subset \fix(\W)$. In fact, since $\W$ is irreducible, we can conclude from the Perron-Frobenius theorem \cite{meyer2000matrix} that $ \fix(\W) = \mathrm{span}\{\e\}$. Consequently, to show that $\cN(\A) \cap \fix(\W)=\{\ze\}$, it suffices to check $\A \e  \neq \ze$. It is clear that $\A \e  \neq \ze$ for inpainting and deblurring. Moreover, $\A$ is a lowpass blur followed by downsampling in superresolution, so $\A \e  \neq \ze$. Thus, $\cN(\A) \cap \fix(\W)=\{\ze\}$ for inpainting, deblurring, and superresolution. It follows from Thm.~\ref{thm:SC} that \eqref{eq:recon} with $\varphi=\varphi_{\W}$  is strongly convex on $(\Re^n, \langle \cdot, \cdot \rangle_2)$. 
\end{proof}

\subsection{Kernel denoiser}

We next extend Thm.~\ref{thm:SC} to kernel denoisers of the form $\W = \D^{-1} \K$. As noted earlier, $\W$ is generally nonsymmetric, and it follows from Thm.~\ref{Moreau} that  $\W$ is not a proximal operator on $(\Re^n, \langle \cdot, \cdot \rangle_2)$. However, we can write  
\begin{equation}
\label{def:Ws}
\W = \D^{-\frac{1}{2}} \W_s \D^{\frac{1}{2}}, \quad \W_s := \D^{-\frac{1}{2}} \K \D^{-\frac{1}{2}}.
\end{equation}
That is, $\W$ is similar to $\W_s \in \Sy^n$. Hence, $\sigma(\W_s) = \sigma(\W) \subset [0,1]$. We know from Proposition~\ref{prop:formula_prox} that $\W_s$ is the proximal operator of ${\varphi_{\W_s}}$ given by \eqref{formula_prox}. We use this to construct the regularizer associated with $\W$.

\begin{proposition}
\label{prop:proxkernelW}
Let $\W$ be a kernel denoiser. Define $\W_s$ as in \eqref{def:Ws} and the corresponding regularizer $\varphi_{\W_s}$ in \eqref{formula_prox}. Define $\phi_{\W} \in \Gamma_0(\Re^n)$ to be
\begin{equation}
\label{def:proxWker}
\phi_{\W}(\x) =  \varphi_{\W_s}(\D^{\frac{1}{2}} \x).
\end{equation}
Then {$\W = \prox_{\phi_\W, \| \cdot \|_{\D}}$}, where $\| \cdot \|_{\D}$ is induced by \eqref{ip}. In other words, $\W$ is a proximal operator on $(\Re^n, \langle \cdot , \cdot \rangle_{\D})$.
\end{proposition}

\begin{proof}
Fix $\x \in \Re^n$. Following \eqref{def:proxWker} and Definition~\ref{defProx}, we need to verify that $\W\x$ is the minimizer of $h \in \Gamma_0(\Re^n)$ given by
\begin{align*}
h(\z) = \frac{1}{2} \| \D^{\frac{1}{2}}(\z-\x)\|_2^2 + \varphi_{\W_s}(\D^{\frac{1}{2}} \z),
\end{align*}
where we have used the fact that $\| \cdot \|_{\D}=\| \D^{\frac{1}{2}} \cdot \|_2$. Performing the substitution $\y=\D^{\frac{1}{2}} \z$, we get
\begin{equation}
\label{eq1}
\underset{\z \in \Re^n}{\arg \min} \ h(\z) = \D^{-\frac{1}{2}} \Big(\underset{\y \in \Re^n}{\arg \min} \ g(\y) \Big),
\end{equation}
where
\begin{equation*}
g(\y):= h(\D^{-\frac{1}{2}} \y) = \frac{1}{2} \| \y- \D^{\frac{1}{2}} \x\|_2^2 + \varphi_{\W_s}(\y).
\end{equation*}
However, from Proposition~\ref{prop:formula_prox} and Definition~\ref{defProx}, we have 
\begin{equation*}
\underset{\y \in \Re^n}{\arg \min} \ g(\y) =  \prox_{\phi_{\W},\| \cdot \|_2}(\D^{\frac{1}{2}} \x)= \W_s (\D^{\frac{1}{2}} \x).
\end{equation*}
The desired result follows from \eqref{def:Ws} and \eqref{eq1}.
\end{proof}

Having identified the regularizer associated with a kernel denoiser, we next analyze when the optimization problem is strongly convex on $(\Re^n, \langle \cdot, \cdot \rangle_{\D})$.

\begin{theorem}
\label{thm:SCnonsym}
Let $\A \in \Re^{m \times n}$ and $\W \in \Re^{n \times n}$ be a kernel denoiser. For $\lambda > 0$, let $f \in \Gamma_0(\Re^n)$ be 
\begin{equation}
\label{def:f}
f(\x) =\ell_{\A}(\x) + \lambda \, \phi_{\W}(\x).
\end{equation}
Then $f$ is strongly convex on $(\Re^n, \langle \cdot, \cdot \rangle_{\D})$ if and only if 
$\cN(\A)  \cap  \fix (\W) = \{\ze\}$.
\end{theorem}

\begin{proof}
We can deduce this from Thm.~\ref{thm:SC}. Indeed, note that \eqref{def:f} is strongly convex on $(\Re^n, \langle \cdot, \cdot \rangle_{\D})$ if and only if $g(\x):=f(\D^{-\frac{1}{2}} \x)$ is strongly convex on $(\Re^n, \langle \cdot, \cdot \rangle_2)$. Now, it follows from \eqref{loss} and \eqref{def:proxWker} that
\begin{align*}
g(\x) & = \ell_{\A}(\D^{-\frac{1}{2}} \x) + \lambda \, \phi_{\W}(\D^{-\frac{1}{2}} \x) \\
& = \ell_{\A \D^{-1/2}}(\x) + \lambda \, \varphi_{\W_s}(\x). 
\end{align*}

We are now in the setting of Thm.~\ref{thm:SC}. In particular,  $g$ is strongly convex on $(\Re^n, \langle \cdot, \cdot \rangle_2)$ if and only if 
\begin{equation}
\label{eq2}
\cN(\A \D^{-\frac{1}{2}}) \cap \fix(\W_s)=\{\ze \}.
\end{equation}
However, $\cN(\A \D^{-\frac{1}{2}})  = \D^{\frac{1}{2}} \cN(\A)$. Moreover, we have from \eqref{def:Ws} that $\fix(\W_s) = \D^{\frac{1}{2}}  \fix(\W)$.
Consequently,
\begin{align*}
\cN(\A \D^{-\frac{1}{2}}) \cap \fix(\W_s) = \D^{\frac{1}{2}}  \big(\cN(\A) \cap  \fix(\W) \big).
\end{align*}
Since $\cN(\A)  \cap  \fix (\W) = \{\ze\}$ by assumption, this verifies \eqref{eq2} and completes the proof.
\end{proof}

Similar to Corollary~\ref{cor:sym}, we can deduce the following result from Thm.~\ref{thm:SCnonsym}.  

\begin{corollary}
\label{cor:kernel}
Let $\W$ be a kernel denoiser and $\A$ correspond to inpainting, deblurring, or superresolution. Then  \eqref{eq:recon} with $\ell=\ell_{\A}$ and $\varphi=\phi_{\W}$ is strongly convex on $(\Re^n, \langle \cdot, \cdot \rangle_{\D})$.
\end{corollary}

\section{Convergence Analysis}
\label{convg}

We now explain how Corollaries~\ref{cor:sym} and \ref{cor:kernel} can be used to establish convergence of any PnP algorithm derived from a proximal method. We have seen that the updates in \eqref{eq:pnp-ista} correspond to running ISTA in $(\Re^n, \langle \cdot, \cdot \rangle_2)$ for symmetric denoisers. On the other hand, if we replace the gradient in \eqref{eq:pnp-ista} with \eqref{gradD}, then PnP-ISTA corresponds to running classical ISTA in $(\Re^n, \langle \cdot, \cdot \rangle_{\D})$ for kernel denoisers. Henceforth, we will not distinguish between symmetric and kernel denoisers, assuming that the gradient is calculated as explained above. The following is the main result of this note.

\begin{theorem}
\label{thm:convg}
Let $\W$ be a linear denoiser as in Corollaries~\ref{cor:sym} and \ref{cor:kernel}, and let $\A$ correspond to inpainting, deblurring, or superresolution. Then any PnP algorithm, with $\W$ as regularizer, corresponds to solving a convex optimization problem of the form \eqref{eq:recon}. Moreover, we can guarantee objective and iterate convergence of the PnP iterates $\{\x_k\}$, i.e., $\ell(\x_k) + \lambda \varphi(\x_k) \to \nu$ and $\x_k \to \x^*$, where $\nu$ and $\x^*$ are the minimum value and the (unique) minimizer of \eqref{eq:recon}.
\end{theorem}

Thm.~\ref{thm:convg} can be deduced from Corollaries~\ref{cor:sym} and \ref{cor:kernel} and standard objective convergence results for proximal algorithms. The proof is deferred to the Appendix. 

We note that in our prior work \cite{ACK2023-contractivity,athalye2023corrections}, iterate convergence was established for PnP-ISTA and PnP-ADMM, but the focus was on symmetric denoisers. On the other hand, Thm.~\ref{thm:convg} provides a unified analysis for symmetric and kernel denoisers, extending the results in \cite{ACK2023-contractivity}. Moreover, the convergence guarantee in Thm.~\ref{thm:convg} covers any PnP algorithm in principle, not just PnP-ISTA and PnP-ADMM.  We remark that the condition $\cN(\A) \cap \fix(\W) =\{\ze\}$ in Corollaries~\ref{cor:sym} and \ref{cor:kernel} was used in \cite{ACK2023-contractivity} to deduce linear convergence of the iterates of PnP-ISTA and PnP-ADMM.

\section{Numerical Validation}
\label{exp}

{We verify the strong convexity of $f$ in \eqref{def:f} using inpainting ($30\%$ random samples) and deblurring ($7\times 7$ uniform blur) experiments, where $\W$ is the symmetric DSG-NLM denoiser \cite{sreehari2016plug}.
The exact computation of the strong convexity index $\mu$ in \eqref{def:SC} will require the full basis information of $\mathcal{R}(\W)$, which is computationally infeasible in our case as $\W$ has an enormous size and a high rank. As an alternative, we have opted
for a lower bound on $\mu$. This was found to be $0.01$ for inpainting and $0.0005$ for deblurring. Details on the numerical computation of $\mu$ are provided in the Appendix. For visualization purposes, plots of the section $t \mapsto g(t \v_0)$ are shown in Fig.~\ref{fig:cvx}, where $\v_0 \in \cR(\W)$ is chosen randomly. The convex behavior of the section serves as a direct indication of the strong convexity of $f$, even though $\mu$ is underestimated~\cite[Thm. 5.27]{beck2017book}. We remark that the $g(t\v_0)$ is a quadratic function of $t$ for this choice of $\v_0$ (see~\eqref{formula_prox}). }
\begin{figure}[!t]
    \centering
    \includegraphics[width=0.7\linewidth]{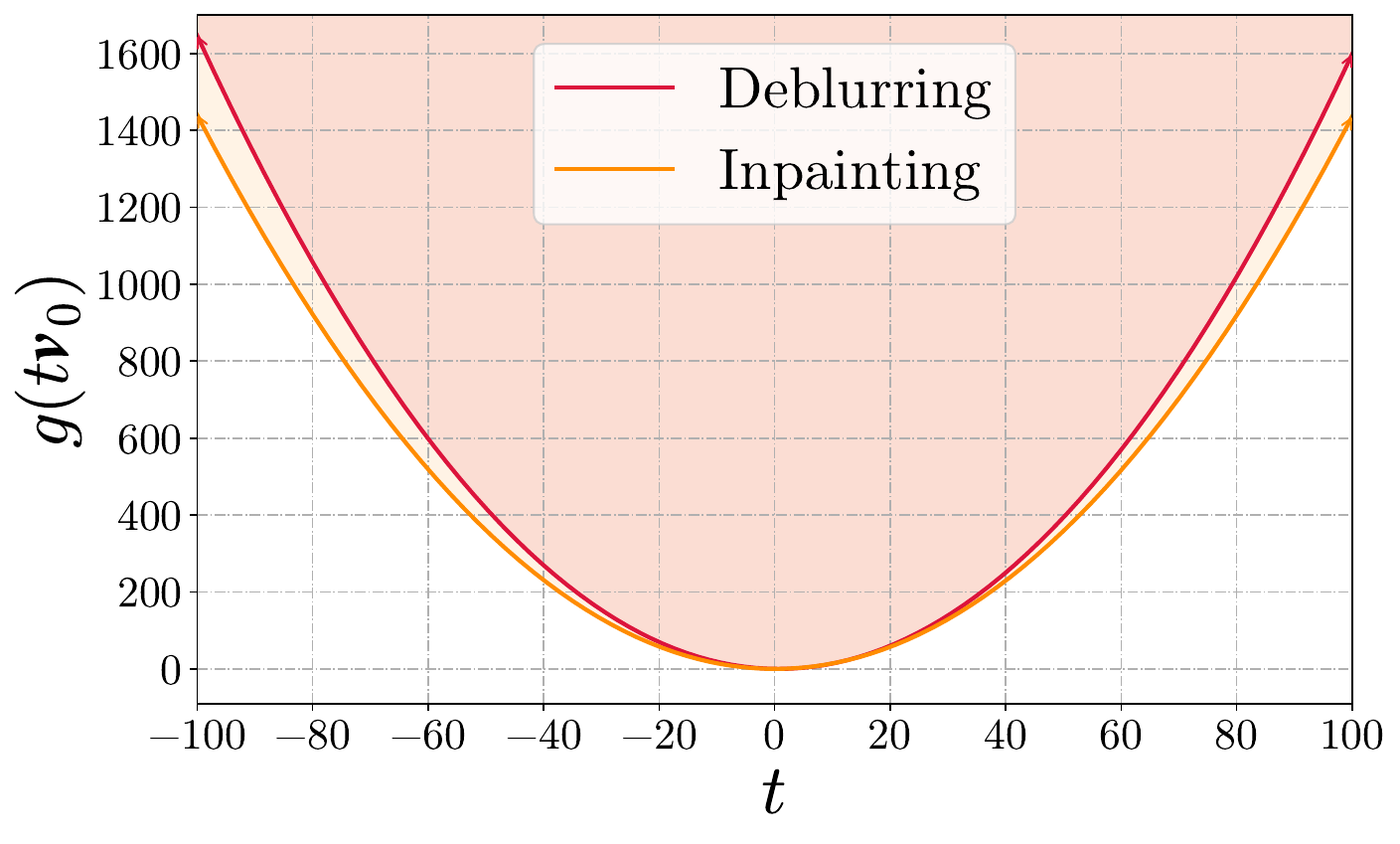}
    \caption{Evidence of the convexity of $g(t \v_0)= f(t \v_0) - (\mu/2)\|t \v_0\|_2^2$ for the inpainting and deblurring problems in Sec.~\ref{exp}. }
    \label{fig:cvx}
\end{figure}

\begin{figure}[t]
\centering
\subfloat[observed]{\label{fig:a}\includegraphics[width = 0.3\textwidth]{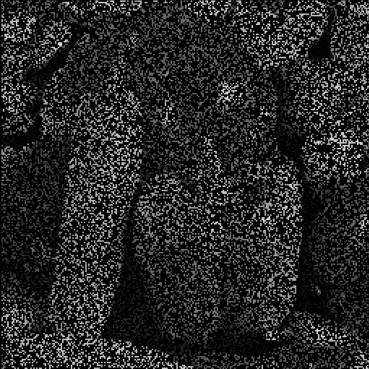}}
\hspace{1pt}
\subfloat[DSG-NLM]{\label{fig:b}\includegraphics[width = 0.3\textwidth]{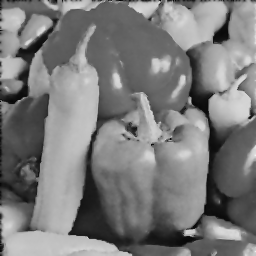}}
\hspace{1pt}
\subfloat[NLM]{\label{fig:d}\includegraphics[width = 0.3\textwidth]{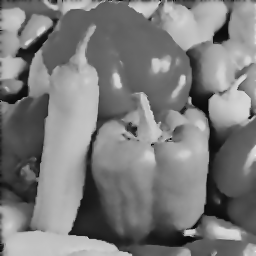}}\\

\includegraphics[width = 1.0\textwidth]{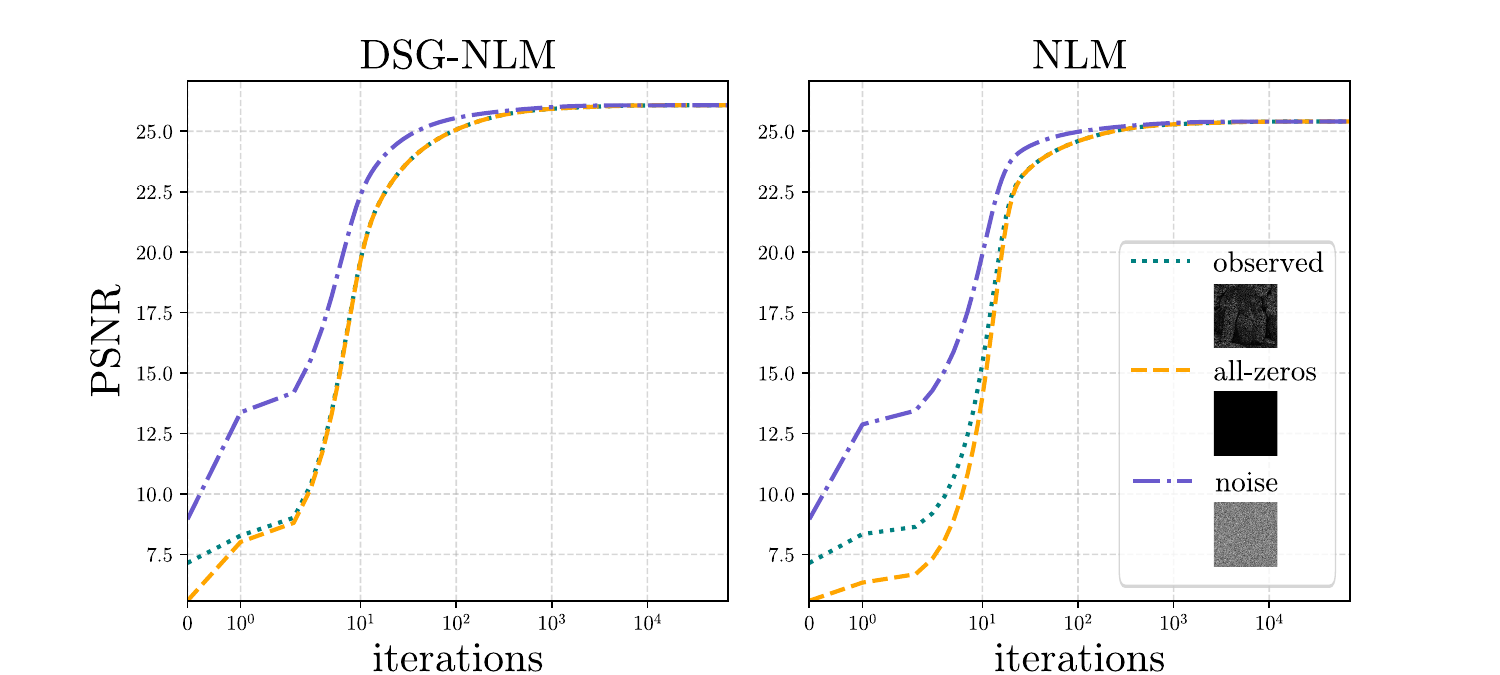}
\caption{Inpainting results for PnP-FISTA using symmetric DSG-NLM and nonsymmetric NLM denoisers. Regardless of the initialization $\x_0$ (indicated in the legend), the iterations converge to the same reconstruction having PSNR values of $26.08$ dB and $25.40$ dB.}
\label{fig:global}
\end{figure}

We know from Thm.~\ref{thm:convg} that any PnP algorithm (using a linear denoiser) converges to a unique reconstruction for any arbitrary initialization $\x_0$. {We tested this for the inpainting experiment described above, using three different initializations along with both the symmetric DSG-NLM and the nonsymmetric NLM denoisers. We see in Fig.~\ref{fig:global} that the reconstructions are identical irrespective of the initialization. The Python codes for reproducing these results can be found in https://github.com/arghyasinha/PnP-StrongConvexity.}

\section{Conclusion}

We showed that PnP regularization of linear inverse problems using kernel denoisers amounts to solving a strongly convex optimization problem. Consequently, we established objective and iterative convergence for any PnP algorithm applied to linear inverse problems, thereby providing a solid theoretical foundation for these methods. In particular, the current analysis addresses two unresolved issues: the iterate convergence of PnP-FISTA and of scaled PnP algorithms using nonsymmetric denoisers.

\section{Appendix}

\subsection{Proof of Proposition~2.1}

We have to show that, for any $\x \in \Re^n$, $\W\x$ is a minimizer of  the function $h= q+\iota_{\cR(\W)}$, where
\begin{equation*}
q(\z) = \frac{1}{2} \| \z - \x \|_2^2+ \frac{1}{2} \z^\top  \W^{\dagger} (\I-\W) \z.
\end{equation*}
Since $\sigma(\W)\subset [0,1]$, the matrix $\W^{\dagger} (\I-\W) $ is positive semidefinite, so that $q$, and hence $h$, is convex. Consequently, $\W\x$ is a minimizer $h$ if and ony if $\ze \in \partial h(\W\x)$, where $\partial h$ is the subdifferential of $h$. 
To compute $\partial h$, we apply the sum rule for the subdifferential \cite[Thm. 23.8]{Rocka}. This requires us to check that 
\begin{equation}
\label{intersect}
\mathrm{ri}  (\mathrm{dom} \ q) \, \cap  \, \mathrm{ri}   \left(\mathrm{dom} \ \iota_{\cR(\W)} \right) \neq \emptyset,
\end{equation}
where $\mathrm{ri}$ and $\mathrm{dom}$ stand for the relative interior and the domain of a set \cite{Rocka}. However, $\mathrm{dom} \ q = \Re^n$, and
\begin{equation*}
\mathrm{ri}  \left(\mathrm{dom} \ \iota_{\cR(\W)} \right) =  \mathrm{ri}  (\cR(\W))=\cR(\W),
\end{equation*}
so that \eqref{intersect} holds trivially. Applying the sum rule to $h= q + \iota_{\cR(\W)}$, we have, for all $\z \in \Re^n$,
\begin{align}
\label{sumrule}
 \partial h (\z) &=  \nabla  q(\z) + \partial \iota_{\cR(\W)} (\z) \nonumber \\
 &= \z-\x+ \W^{\dagger} (\I-\W) \z +  \partial \iota_{\cR(\W)} (\z) .
\end{align}
We know that $\partial \iota_{\cR(\W)}(\z)$ is the normal cone~ \cite[Example~3.5]{beck2017book} of $\cR(\W)$ for all $\z \in \cR(\W)$. In particular, since $\W$ is symmetric, the normal cone of $\cR({\W})$ is its orthogonal complement $\cN(\W)$. 
Moreover, since $\W^\dagger \W^2 =\W \W^\dagger \W= \W$~\cite{meyer2000matrix}, we have  from \eqref{sumrule} that
\begin{equation}
\label{delh}
\partial h (\W \x)= -(\I-\W^{\dagger}\W) \x+ \cN(\W),
\end{equation}
Applying $\W^\dagger \W^2 = \W$ once more, we get $\W(\I-\W^\dagger\W)\x = \ze$. Thus, $\boldsymbol{0}\in\partial h(\W\x)$, which completes the proof.

\subsection{Proof of Theorem~4.1}

For a linear denoiser, we know that the PnP algorithm corresponds to solving the optimization problem
\begin{equation}
\label{eq:recon}
\nu = \min_{\x \in \Re^n} \, f(\x), \quad f(\x) := \ell_{\A}(\x) + \lambda\, \varphi(\x).
\end{equation}
using a proximal (or scaled proximal) algorithm. It follows from Corollaries~3.1 and 3.2 that when $\A$ corresponds  to inpainting, deblurring, or superresolution, $f$ is strongly convex on $(\Re^n, \langle \cdot, \cdot \rangle)$ for some $ \langle \cdot, \cdot \rangle$. Thus, we can conclude from Thm.~3.1 that $\nu$ is finite, and there exists a unique minimizer $\x^* \in \Re^n$ such that $f(\x^*)=\nu$. 

Objective convergence, i.e., $f(\x_k) \to \nu$, where $\{\x_k\}$ are the iterates of the PnP algorithm, follows from standard results on the (objective) convergence of proximal algorithms; e.g., see \cite{beck2017book,parikh2014proximal}. These results implicitly assume that \eqref{eq:recon} has a minimizer, which is guaranteed for our problem.

To prove iterate convergence, it suffices to show that the (subsequential) limit points of the sequence $\{\x_k\}$ are identical. Recall from \cite{rudin1976principles} that $\bar{\x} \in\Re^n$ is called a subsequential limit of the sequence $\{\x_k\}$ if there exists a subsequence $\{\x_{k_i}\}$ of $\{\x_k\}$ that converges to $\bar{\x}$. It is evident that $\{\x_k\}$ converges to some $\x^* \in \Re^n$ if and only if all its subsequences are convergent and their limits equal $\x^*$.

We begin with the observation that the PnP iterates $\{\x_k\}$ are bounded. Indeed, since $f$ is strongly convex, it is coercive \cite{beck2017book}, i.e., $f(\x) \to \infty$ as $\|\x\| \to \infty$, where $\|\cdot\|$ is the norm induced by $\langle \cdot, \cdot \rangle$. Thus, if $\{\x_k\}$ were not bounded, we would end up contradicting the fact that $\{f(\x_k)\}$ is convergent.

As $\{\x_k\}$ is bounded, we know from the Bolzano-Weierstrass theorem that it has at least one subsequential limit  \cite{rudin1976principles}. In fact, we claim that $\{\x_k\}$ has exactly one such limit point, namely, the minimizer $\x^*$ of \eqref{eq:recon}. Indeed, suppose $\bar{\x}$ is a subsequential limit of  $\{\x_k\}$. Then there exists a subsequence $\{\x_{k_i}\}$ such that $\x_{k_i} \to \bar{\x}$. However, since $f$ is lower semicontinuous, we have
\begin{equation*}
\nu \leqslant f(\bar{\x}) \leqslant \lim_{k \to \infty} \, f(\x_{k_i}) =  \lim_{k \to \infty} \, f(\x_{k}) =\nu,
\end{equation*}
that is, $\bar{\x}$ is a minimizer of \eqref{eq:recon}. Finally, since $f$ has exactly one minimizer, it must be that $\bar{\x}=\x^*$. This holds for any subsequential limit $\bar{\x}$, whereby we it follows that $\{\x_k\}$ converges to $\x^*$.

\subsection{Computation of $\mu$}

We discuss how the strong convexity index $\mu$ is computed in the experiments in Sec.~V. Determining $\mu$ exactly is not easy as it requires us to solve an optimization over a nonconvex unit sphere $\mathbb{V} \cap \, \mathcal{S}_n$ (see Thm.~3.2). Nonetheless, we can establish a lower bound for $\mu$ by expanding the domain to $\mathcal{S}_n$. The lower bound in question (which turns out to be positive as reported in Sec.~V) is simply the smallest eigenvalue $\lambda_{\min}(\Q)$, where
\begin{equation*}
\Q=\A^\top\!\A + \lambda \, \W^\dagger(\I - \W).
\end{equation*}

The computational challenge here is that neither the forward model $\A$ nor the denoiser $\W$ cannot be stored and manipulated as matrices; we have to work with $\A, \A^\top \W$ and $\W^\dagger$ as operators (black boxes). This is where we can utilize the power method \cite{borm2012numerical}, which involves repeatedly applying $\Q$ to determine its dominant eigenvalue $d(\Q)$, i.e., the eigenvalue with the largest absolute value. We apply the power method twice to find $\lambda_{\min}(\Q)$. The first application is to identify the dominant eigenvalue $d(\Q)$, which in our case is $d(\Q)=\lambda_{\max}(\Q) \geqslant 0$ since $\Q$ is positive semidefinite. We next apply the power method on the shifted operator $\Q_s = \Q - d(\Q) \I$ to get $d(\Q_s)$. We can verify that $d(\Q_s)=\lambda_{\min}(\Q) - d(\Q)$, which gives us
\begin{equation*}
\lambda_{\min}(\Q) = d(\Q_s) + d (\Q).
\end{equation*}

In addition to the operators \(\A\), \(\A^\top\), and \(\W\), which can all be applied efficiently, we also require the pseudoinverse \(\W^\dagger\) to compute \(\Q\). This can be achieved through multiple applications of \(\W\).

\bibliographystyle{unsrt}
\bibliography{refs}

\begin{thebibliography}{10}

\bibitem{sreehari2016plug}
S.~Sreehari, S.~V. Venkatakrishnan, B.~Wohlberg, G.~T. Buzzard, L.~F. Drummy,
  J.~P. Simmons, and C.~A. Bouman.
\newblock Plug-and-play priors for bright field electron tomography and sparse
  interpolation.
\newblock {\em IEEE Trans. Comput. Imag.}, 2(4):408--423, 2016.

\bibitem{hurault2022gradient}
Samuel Hurault, Arthur Leclaire, and Nicolas Papadakis.
\newblock Gradient step denoiser for convergent plug-and-play.
\newblock {\em Proc. ICLR}, 2022.

\bibitem{cohen2021regularization}
R.~Cohen, M.~Elad, and P.~Milanfar.
\newblock Regularization by denoising via fixed-point projection ({RED-PRO}).
\newblock {\em SIAM J. Imaging Sci.}, 14(3):1374--1406, 2021.

\bibitem{romano2017little}
Y.~Romano, M.~Elad, and P.~Milanfar.
\newblock The little engine that could: {Regularization} by denoising ({RED}).
\newblock {\em SIAM J. Imaging Sci.}, 10(4):1804--1844, 2017.

\bibitem{reehorst2018regularization}
E.~T. Reehorst and P.~Schniter.
\newblock Regularization by denoising: {Clarifications} and new
  interpretations.
\newblock {\em IEEE Trans. Comput. Imag.}, 5(1):52--67, 2018.

\bibitem{zhu2023denoising}
Yuanzhi Zhu, Kai Zhang, Jingyun Liang, Jiezhang Cao, Bihan Wen, Radu Timofte,
  and Luc Van~Gool.
\newblock Denoising diffusion models for plug-and-play image restoration.
\newblock {\em Proc. IEEE Conf. Comp. Vis. Pattern Recognit.}, pages
  1219--1229, 2023.

\bibitem{terris2024equivariant}
Matthieu Terris, Thomas Moreau, Nelly Pustelnik, and Julian Tachella.
\newblock Equivariant plug-and-play image reconstruction.
\newblock {\em Proc. CVPR}, pages 25255--25264, 2024.

\bibitem{pesquet_learning_2021}
Jean-Christophe Pesquet, Audrey Repetti, Matthieu Terris, and Yves Wiaux.
\newblock Learning maximally monotone operators for image recovery.
\newblock {\em SIAM J. Imaging Sci.}, 14(3):1206--1237, 2021.

\bibitem{hertrich_convolutional_2021}
Johannes Hertrich, Sebastian Neumayer, and Gabriele Steidl.
\newblock Convolutional proximal neural networks and plug-and-play algorithms.
\newblock {\em Linear Algebra and its Applications}, 631:203--234, 2021.

\bibitem{goujon_learning_2024}
Alexis Goujon, Sebastian Neumayer, and Michael Unser.
\newblock Learning weakly convex regularizers for convergent
  image-reconstruction algorithms.
\newblock {\em SIAM J. Imaging Sci.}, 17(1):91--115, 2024.

\bibitem{CWE2017}
S.~H. Chan, X.~Wang, and O.~A. Elgendy.
\newblock Plug-and-play {ADMM} for image restoration: {Fixed-point} convergence
  and applications.
\newblock {\em IEEE Trans. Comput. Imag.}, 3(1):84--98, 2017.

\bibitem{Ryu2019_PnP_trained_conv}
E.~Ryu, J.~Liu, S.~Wang, X.~Chen, Z.~Wang, and W.~Yin.
\newblock Plug-and-play methods provably converge with properly trained
  denoisers.
\newblock {\em Proc. Intl. Conf. Mach. Learn.}, 97:5546--5557, 2019.

\bibitem{Teodoro2019PnPfusion}
A.~M. Teodoro, J.~M. Bioucas-Dias, and M.~A.~T. Figueiredo.
\newblock A convergent image fusion algorithm using scene-adapted
  {Gaussian}-mixture-based denoising.
\newblock {\em IEEE Trans. Image Process.}, 28(1):451--463, 2019.

\bibitem{proximal-hurault22a}
Samuel Hurault, Arthur Leclaire, and Nicolas Papadakis.
\newblock Proximal denoiser for convergent plug-and-play optimization with
  nonconvex regularization.
\newblock {\em Proc. ICML}, 2022.

\bibitem{nair2024averaged}
Pravin Nair and Kunal~N Chaudhury.
\newblock Averaged deep denoisers for image regularization.
\newblock {\em J. Math. Imaging Vis.}, pages 1--18, 2024.

\bibitem{milanfar2013tour}
P.~Milanfar.
\newblock {A tour of modern image filtering: new insights and methods, both
  practical and theoretical}.
\newblock {\em IEEE Signal Process. Mag.}, 30(1):106--128, 2013.

\bibitem{milanfar2013symmetrizing}
P.~Milanfar.
\newblock Symmetrizing smoothing filters.
\newblock {\em SIAM J. Imaging Sci.}, 6(1):263--284, 2013.

\bibitem{nair2019hyperspectral}
P.~Nair, V.~S. Unni, and Kunal~N Chaudhury.
\newblock Hyperspectral image fusion using fast high-dimensional denoising.
\newblock {\em Proc. IEEE Intl. Conf. Image Process.}, pages 3123--3127, 2019.

\bibitem{nair2022plug}
Pravin Nair and Kunal~N Chaudhury.
\newblock Plug-and-play regularization using linear solvers.
\newblock {\em IEEE Trans. Image Process.}, 31:6344--6355, 2022.

\bibitem{gavaskar2021plug}
Ruturaj~G Gavaskar, Chirayu~D Athalye, and Kunal~N Chaudhury.
\newblock On plug-and-play regularization using linear denoisers.
\newblock {\em IEEE Trans. Image Process.}, 30:4802--4813, 2021.

\bibitem{nair2021fixed}
Pravin Nair, Ruturaj~G Gavaskar, and Kunal~Narayan Chaudhury.
\newblock Fixed-point and objective convergence of plug-and-play algorithms.
\newblock {\em IEEE Trans. Comput. Imag.}, 7:337--348, 2021.

\bibitem{beck2017book}
A.~Beck.
\newblock {\em First-Order Methods in Optimization}.
\newblock {SIAM}, 2017.

\bibitem{bouman2022foundations}
Charles~A Bouman.
\newblock {\em Foundations of Computational Imaging: A Model-Based Approach}.
\newblock SIAM, 2022.

\bibitem{buades2005non}
A.~Buades, B.~Coll, and J.~M. Morel.
\newblock A non-local algorithm for image denoising.
\newblock {\em Proc. IEEE Conf. Comp. Vis. Pattern Recognit.}, 2:60--65, 2005.

\bibitem{takeda2007kernel}
H.~Takeda, S.~Farsiu, and P.~Milanfar.
\newblock Kernel regression for image processing and reconstruction.
\newblock {\em IEEE Trans. Image Process.}, 16(2):349--366, 2007.

\bibitem{talebi2013global}
H.~Talebi and P.~Milanfar.
\newblock Global image denoising.
\newblock {\em IEEE Trans. Image Process.}, 23(2):755--768, 2013.

\bibitem{Moreau1965}
J.~J. Moreau.
\newblock Proximit\'{e} et dualit\'{e} dans un espace {H}ilbertien.
\newblock {\em Bull. Soc. Math. France}, 93:273--299, 1965.

\bibitem{ACK2023-contractivity}
C.~D. Athalye, K.~N. Chaudhury, and B.~Kumar.
\newblock On the contractivity of plug-and-play operators.
\newblock {\em IEEE Signal Process. Lett.}, 30:1447--1451, 2023.

\bibitem{meyer2000matrix}
C.~D. Meyer.
\newblock {\em {Matrix Analysis and Applied Linear Algebra}}.
\newblock SIAM, 2000.

\bibitem{athalye2023corrections}
Chirayu~D Athalye and Kunal~N Chaudhury.
\newblock Corrections to “{O}n the contractivity of plug-and-play
  operators”.
\newblock {\em IEEE Signal Process. Lett.}, 30:1817--1817, 2023.

\bibitem{Rocka}
Ralph~Tyrell Rockafellar.
\newblock {\em Convex Analysis}.
\newblock Princeton University Press, 1970.

\bibitem{parikh2014proximal}
N.~Parikh and S.~Boyd.
\newblock Proximal algorithms.
\newblock {\em Found. Trends. Optim.}, 1(3):127--239, 2014.

\bibitem{rudin1976principles}
Walter Rudin.
\newblock {\em Principles of Mathematical Analysis}.
\newblock McGraw-Hill, 1976.

\bibitem{borm2012numerical}
Steffen B{\"o}rm and Christian Mehl.
\newblock {\em Numerical Methods for Eigenvalue Problems}.
\newblock Walter de Gruyter, 2012.

\end{thebibliography}

\end{document}